\documentclass[12pt]{amsart}
\usepackage[T1]{fontenc}
\usepackage{geometry}                % See geometry.pdf to learn the layout options. There are lots.
\usepackage{graphicx}
\usepackage{amsfonts,latexsym,rawfonts,amsmath,amssymb,amsthm}
\usepackage{amsrefs}
\usepackage{epstopdf}
\usepackage{cite}

\usepackage[plainpages=false]{hyperref}
\RequirePackage{color}

\newtheorem{theorem}{Theorem}[section]

\newtheorem{definition}[theorem]{Definition}

\newtheorem{lemma}[theorem]{Lemma}

\newtheorem{remark}[theorem]{Remark}

\numberwithin{equation}{section}

\newcommand{\beq}{\begin{equation}}
\newcommand{\eeq}{\end{equation}}
\newcommand{\beqo}{\begin{equation*}}
\newcommand{\eeqo}{\end{equation*}}
\newcommand{\beqs}{\begin{eqnarray*}}
\newcommand{\eeqs}{\end{eqnarray*}} 
\newcommand{\beqn}{\begin{eqnarray}}
\newcommand{\eeqn}{\end{eqnarray}}

\def\rho{\varrho}

\def\dist{\text{dist}}

\def\tr{\text{trace}}
\def\loc{\text{loc}}

\RequirePackage{doi}

\begin{document}
\title{$C^{1,\alpha}$ regularity of the solution for the obstacle problem for the linearized Monge-Amp\`ere operator}

\author[M. Ji]{Meng Ji}
\address{School of Mathematics and Applied Statistics,
University of Wollongong\\
Northfields Ave\\
2522 NSW, Australia}
\email{\href{mailto:mengj@uow.edu.au}{mengj@uow.edu.au}}

%\keywords{}

\maketitle

\baselineskip18pt

%%%%%%%%%%%%%%%%%%%%%%%%%%%%%%%%%%%%%%%%%%%%%%%%%%%%%%%
\begin{abstract}
In this paper, we study the regularity of the solution for the obstacle problem associated with the linearized Monge-Amp\`ere operator:
\begin{align*}
    \begin{cases}
        &u\geq\varphi \text{\quad in } \Omega\\
        &L_{ w}u=\tr( W D^{2}u)\leq 0 \text{\quad in } \Omega\\
        &L_{ w}u= 0 \text{\quad in } \{u>\varphi\}\\
        &u=0 \text{\quad on } \partial\Omega,
    \end{cases}
\end{align*}
where $ W=(\det D^{2} w) D^{2} w^{-1}$ is the matrix of cofactor of $D^{2} w$, $w$ satisfies $\lambda \leq \det D^{2} w \leq \Lambda$  and $ w=0$ on $\partial \Omega$, $\varphi$ is the obstacle with at least $C^{2}(\bar{\Omega})$ smoothness, $\Omega$ is an open bounded convex domain. We show the existence and uniqueness of an $L^{n}$-viscosity solution by using Perron’s method and the comparison principle. Our primary result is to prove that the solution exhibits local $C^{1,\gamma}$ regularity for any $\gamma \in (0,1)$, provided that it is a strong solution in $W^{2,n}_{\text{loc}}(\Omega)$. 
\end{abstract}

\section{Introduction}
\subsection{Motivation.} The obstacle problem can be traced back to work in potential theory and elasticity, such as the study of membranes stretched over obstacles. This type of problem originated from Stampacchia’s work in 1964 in \cite{S64}. The classical obstacle problem is to find the equilibrium position of an elastic membrane whose boundary is held fixed and which is constrained to lie above a given obstacle. Specifically, the problem is to find the solution $u$ that satisfies
\begin{equation}\label{OP}
    \begin{cases}
        &u\geq\varphi \text{\quad in } \Omega\\
        &L u\leq 0 \text{\quad in } \Omega\\
        &L u= 0 \text{\quad in } \{u>\varphi\}\\
        &u=0 \text{\quad on } \partial\Omega,
    \end{cases}
\end{equation}
where $u$ denotes the position of the membrane, $L=\Delta$ is the Laplacian operator, $\varphi$ is the obstacle and $\Omega$ is an open and bounded domain in $\mathbb{R}^{n}$. Two main questions for the obstacle type problem are the regularity of the solution and the regularity of the free boundary, which is $\partial \{u>\varphi\}\cap\Omega$. For the classical obstacle problem, Frehse in \cite{F72} first showed that the optimal regularity of the solution $u$ is $C_{\loc}^{1,1}$. For the regularity of the free boundary, Caffarelli in \cite{C77} proved that it is $C^{1,\alpha}$ near regular points, and hence it is analytic by the result from Kinderlehrer and Nirenberg (see \cite{KN77}). For singular points, Figalli and Serra showed that the free boundaries are locally $C^{1,\alpha}$ in \cite{FS19}.

Over the past decades, obstacle problems have captivated scientists due to their rich application in multiple disciplines. They naturally appear in contexts of fluid mechanics (e.g., flow through porous media), mathematical finance (e.g., optimal stopping in option pricing), and material science (e.g., superconductivity models). We refer \cite{CF78}, \cite{PSU12}, \cite{R87} to those readers who are interested in these topics. Furthermore, from a mathematical perspective, the obstacle problem for different operators $L$ in \eqref{OP}, especially for elliptic operators, has also attracted increasing attention of many mathematicians, see \cite{L98}, \cite{LS01}, \cite{S05}, and \cite{HTW24} for instance.

Although there is a great amount of literature on the obstacle problem, most of their work focuses on uniformly elliptic operators. Obstacle problems for non-uniformly elliptic operators are much less studied. One of the main reasons is that degeneracy may lead to less regularity of the solution $u$, which may further cause less regularity of the free boundary. Therefore, this paper aims to study the obstacle problem for the linearized Monge-Amp\`ere operator, which is a linear and elliptic operator, possibly degenerate. 

The linearized Monge-Amp\`ere operator is an important tool that bridges the nonlinear Monge-Amp\`ere equation to the tractable realm of linear PDEs. It usually appears in the context of affine geometry (see \cite{C79}), Abreu’s equation (see \cite{A98}), and semigeostrophic equations in fluid mechanics (see \cite{CNP91}). Mathematically, the linearized Monge-Amp\`ere operator is defined by 
\begin{align*}
    L_{ w}u= \tr ( W D^{2}u)=  W_{ij}u_{ij},
\end{align*}
where $ w$ satisfies the Monge-Amp\`ere equation
\begin{align*}
    \det D^{2} w=f, \text{\quad} 0< \lambda \leq f \leq \Lambda,
\end{align*}
and $ W = ( W_{ij}) = (D^{2} w)^{-1}(\det D^{2} w)$ is the cofactor matrix of $D^{2} w$. Even though we know that the product of the eigenvalues of the coefficient matrix $W$ is bounded, that is, 
$$ 
\lambda^{n-1} \leq \det W \leq \Lambda^{n-1},
$$
the eigenvalues of $W$ are not necessarily bounded away from $0$ and $\infty$. Therefore, the linearized Monge-Amp\`ere operator may be degenerate, which is the main challenge when we establish the regularity results of the solution.

\subsection{Statement of the main results}
In this subsection, we present the main results of the obstacle problem for the linearized Monge-Amp\`ere operator. Specifically, we study the solution $u$ of the equations
\begin{equation}\label{LMA}
    \begin{cases}
        &u\geq\varphi \text{\quad in } \Omega\\
        &L_{ w}u=\tr( W D^{2}u)\leq 0 \text{\quad in } \Omega\\
        &L_{ w}u= 0 \text{\quad in } \{u>\varphi\}\\
        &u=0 \text{\quad on } \partial\Omega,
    \end{cases}
\end{equation}
where $ W$ is the cofactor matrix of $D^{2} w$ defined as above and
\begin{equation}\label{condition}
    \begin{cases}
        &\Omega\subset\mathbb{R}^{n} \text{ is open, bounded and convex}\\
        &\varphi \text{ is at least }C^{2}(\bar{\Omega}), \varphi<0 \text{ on } \partial\Omega, \text{ and }\varphi(x)>0 \text{ for some } x\in\Omega\\
        & w\in C(\bar{\Omega}), w=0 \text{ on }\partial\Omega, \text{ and }\lambda\leq\det(D^{2} w)=f\leq\Lambda \text{ in } \Omega, \text{ where } f\in C(\Omega)\\
    \end{cases}
\end{equation}
By Wang's counterexample in \cite{W95}, the condition for $ w$ is sharp if we want to show $C_{\loc}^{1,\gamma}(\Omega)$ regularity, and $\varphi$ and $\Omega$ are general settings in the obstacle-type problem. Here are our main results.
\begin{theorem}\label{EU}
    Suppose \eqref{condition} holds, then there is a unique $L^{n}$-viscosity solution $u\in C(\bar\Omega)$ for \eqref{LMA}.
\end{theorem}

We will prove this theorem via Perron's method. That is, we will show the existence of the solution by finding the least $L^{n}$-supersolution that satisfies \eqref{LMA}. The uniqueness follows from the comparison principle of the elliptic operator. Alternatively, one can prove the existence and uniqueness via the penalization method (e.g., see \cite{L98}).

If we further assume that the solution is a strong solution, namely $u\in W_{\loc}^{2,n}(\Omega) \cap C(\bar{\Omega})$, then we have the following regularity result.
\begin{theorem}\label{main}
    Suppose \eqref{condition} holds, and $u\in W_{\loc}^{2,n}(\Omega)\cap C(\bar{\Omega})$ is a solution to \eqref{LMA}, then $u\in C_{\loc}^{1,\gamma}(\Omega)$ for any $\gamma\in(0,1)$.
\end{theorem}

The only problem of the regularity is the solution when it crosses the free boundary, since $u$ is controlled by $\varphi$ in $\{u=\varphi\}$, and the interior regularity theory for the linearized Monge-Amp\`ere equation provides $C_{\loc}^{1,\gamma}$ regularity of $u$ in $\{u>\varphi\}$. Here is the strategy of the proof. Without loss of generality, we assume that $\Omega$ is normalized. Then, we focus on a section $S_{h,w}(x_{0})$ around a free boundary point $x_{0} \in \partial \{u>\varphi\}$. The key gradient to show the regularity is the iteration method. Specifically, after normalizing the section $S_{h,w}(x_{0})$, then by the estimate in Lemma \ref{estimate}, $h_{0}^{-\frac{1}{2}}A_{1}S_{h_{0},w}(x_{0})$ is comparable to $B_{\sqrt{2}}$ for some positive definite matrix $A_{1}$. Based on that, Lemma \ref{better estimate} shows that $h_{0}^{-1}A_{2}S_{h_{0}^{2},w}(x_{0})$ is also comparable to $B_{\sqrt{2}}$ for some positive definite matrix $A_{2}$. Applying Lemma \ref{better estimate} inductively, we can get $h_{0}^{-\frac{k}{2}}A_{k}S_{h_{0}^{k},w}(x_{0})$ is comparable to $B_{\sqrt{2}}$. Therefore, after we obtain the bounds for $A_{k}$ for each $k$, we can prove that $Du$ grows at most as $r^{\alpha}$ away from the free boundary, and hence we can obtain the $C^{1, \gamma}$ regularity. 

The paper is organized as follows: In Section 2, we present essential preliminary results, including properties of sections for the Monge-Amp\`ere equation and key findings related to the linearized Monge-Amp\`ere equation. In Section 3, we will give the definition of an $L^{n}$-viscosity solution and some basic properties, such as the comparison principle, and then the proof of Theorem \ref{EU}. In Section 4, we prove Theorem \ref{main} under the assumption that it is a strong solution in $W_{\loc}^{2,n}(\Omega)$.

\textbf{Acknowledgments.} The author would like to thank to his Ph.D. supervisor, Prof. Jiakun Liu, for his invaluable guidance and constant encouragement.

\section{Preliminary results}\label{s2}
We divide this section into two parts. In the first part, we recall the definition and some properties of sections for the Monge-Amp\`ere equation. In the second part, we will provide some results about the linearized Monge-Amp\`ere equation that are relevant to the following parts of this paper.

\subsection{Properties of section}
We start this subsection by recalling an important geometric lemma proposed by John (see \cite{J48}).
\begin{lemma}[John's lemma]\label{John}
    Let $K \subset \mathbb{R}^{n}$ be a convex bounded domain with nonempty interior, then there exists an ellipsoid $E$, such that
    \begin{align*}
        E \subset K \subset nE,
    \end{align*}
where $nE$ denotes the $n$-dilation of $E$. It follows that we can find an affine transformation $T(x)=Ax+b$, where $A$ is an $n \times n$ invertible matrix and $b \in \mathbb{R}^{n}$, such that
    \begin{align*}
        B_{1}(0) \subset T(K) \subset B_{n}(0).
    \end{align*}
We call $T(K)$ the normalization of $K$. A bounded and convex domain $\Omega$ is called a normalized domain if $B_{1}(0)\subset\Omega\subset B_{n}(0)$.
\end{lemma}

Next, we recall the definition and some properties of sections (see \cite{GH00}).

\begin{definition}
    Let $ w\in C^{1}(\Omega)$ be a convex function on a convex domain $\Omega$, a section of $ w$ at $x_{0}$ with height $h$ is defined by:
\begin{align*}
    S_{h, w}(x_{0})=\{x\in\Omega: w(x)\leq w(x_{0})+\nabla w(x_{0})(x-x_{0})+h\}.
\end{align*}
\end{definition}

If a section $S_{h, w}(x_{0})$ has nonempty interior, then by John's lemma, we can find an affine transformation $T$, such that $B_{1}(0)\subset T\big(S_{h, w}(x_{0})\big)\subset B_{n}(0)$. Note that when $w$ satisfies the assumption \eqref{condition}, $w$ is strictly convex by Caffarelli's result in \cite{C90localization}, and hence $S_{h,w}(x_{0})$ has nonempty interior. In the following context, we may use $S_{h}(x_{0})$ to denote the section of $ w$  for convenience. 

\begin{remark}
    Section is a really important tool when we study the regularity of the linearized Monge-Amp\`ere equation. To clarify this, let us start by considering the linear equation:
\[ L(v) := a_{ij} v_{ij} = 0 \quad \text{in} \quad \Omega, \]
where the coefficients \( a_{ij} \) are uniformly elliptic, satisfying:
\[ \gamma_1 |\xi|^2 \leq a_{ij}(x) \xi_i \xi_j \leq \gamma_2 |\xi|^2 \quad \forall \xi \in \mathbb{R}^n, \, x \in \Omega, \]
with some positive constants \( 0 < \gamma_1 \leq \gamma_2 \).
Next, for quadratic polynomials:
\[ P(x) := \frac{1}{2} |x|^2+ \langle a, x \rangle + b, a\in\mathbb{R}^{n}, b\in \mathbb{R},\]
we have \( L(P) \approx 1 \). As \( a \) and \( b \) vary, the sublevel sets of \( P \) represent all balls in \( \mathbb{R}^n \). A central tool in regularity theory for uniformly elliptic equations is the Harnack inequality. It states that if \( L(v) = 0 \) and \( v \geq 0 \) within a ball \( B_{2r}(x) \subset \subset \Omega \), then:
\[ \sup_{B_r(x)} v \leq C \inf_{B_r(x)} v, \]
where \( C \) depends on the dimension \( n \) and the ellipticity constants \( \gamma_1, \gamma_2 \). This is a strong property, since if \( v \) is zero at any point in \( B_r(x) \), it is identically zero throughout the ball. Moreover, a classical consequence of the Harnack inequality is the H\"older continuity of the solution.

Inspired by the above discussion, consider a function $u$ satisfying $L_w(u) = 0$, where $w$ is a solution to the Monge-Amp\`ere equation $\det D^2 w = f$ with $0 < \lambda \leq f \leq \Lambda < \infty$. In particular, for any affine function $\ell$, we have $L_w(w - \ell) = L_w(w) = \det D^2 w = f$. Since $0 < \lambda \leq f \leq \Lambda$, it follows that $L_w(w - \ell) \approx 1$. This suggests that $w - \ell$ plays a role similar to the quadratic polynomials $P$ introduced in the uniformly elliptic equation. In this framework, sections of $w-\ell$, or equivalently, the sublevel sets of $w - \ell$, as $\ell$ varies, are analogues to balls. Thus, one may expect the Harnack inequality, which we will provide in the next subsection, for the linearized Monge-Amp\`ere equation on sections of $w$ instead of balls.
\end{remark}

The following important theorem provides an estimate between the height of a section and the radius of a ball near a point (Theorem 3.3.8 in \cite{G01}).

\begin{theorem}\label{section and ball}
Suppose $\Omega$ is a normalized convex domain, $ w \in C(\overline{\Omega})$ is a convex function such that $\lambda \leq \det D^2  w \leq \Lambda$ in $\Omega$ and $ w = 0$ on $\partial \Omega$ . For any $\Omega' \subset \subset \Omega$, there exist positive constants $t_0$, $C_1$, $C_2$, and $\sigma$ such that for $x_0 \in \Omega'$ and $0 < h \leq t_0$,
\begin{equation}\label{eq SB}
B_{C_1 h}(x_0) \subset S_{h}(x_0) \subset B_{C_2 h^{\sigma}}(x_0),
\end{equation}
where $\sigma = \sigma(n, \lambda, \Lambda)$ and $t_0$, $C_1$, $C_2$ depend only on $n$, $\lambda$, $\Lambda$, and $\dist(\Omega', \partial \Omega)$.
\end{theorem}

\begin{proof}
Choose \( t_0 > 0 \) small enough so that $ S_{t_0}(x_0) \subset \Omega $ for all $ x_0 \in \Omega' $. We claim that there exists $ \beta \in (1/2, 1) $, depending on $ n, \lambda, \Lambda $, such that for $ 0 < h \leq t_0 $,
\begin{equation}\label{half SB}
\frac{1}{2} S_h(x_0) \subset S_{h/2}(x_0) \subset \beta S_h(x_0),
\end{equation}
where \( \beta S_h(x_0) = \{ x_0^{*} + \beta (x - x_0^{*}) : x_{0}^{*} \text{ is the center of mass}, x \in S_h(x_0) \} \). Indeed, the first inclusion holds due to the convexity of $ w$, and the second inclusion follows from Theorem 2.1 in \cite{GH00}.

We first show \eqref{eq SB} for $ h = 2^{-i} t_{0} $ where $i=0,1,2...$. Iteratively applying the inclusion gives
\[
2^{-i} S_{t_{0}}(x_0) \subset S_{h}(x_0) \subset \beta^i S_{t_{0}}(x_0).
\]

Then, the first inclusion in \eqref{eq SB} holds since by Alexandrov estimate (see \cite{GH00}) and Corollary 3.2.4 in \cite{G01}, we have
\begin{equation}\label{half SB iteration}
    \dist (x_{0}, \partial S_{t_{0}}(x_{0})) \geq C t_{0}^{\frac{n}{2}}.
\end{equation}  
Therefore, $B_{C t_{0}^{n/2}}(x_{0}) \subset  S_{t_{0}}(x_{0})$ and it follows that $S_{h}(x_{0}) \supset B_{2^{-i}C t_{0}^{n/2}}(x_{0})=B_{Ct_{0}^{n/2-1}h}$.

For the second inclusion in \eqref{eq SB}, we note that there exists $R>0$ such that $S_{t_{0}}(x_{0}) \subset B_{R}(x_{0}) $ since $S_{t_{0}}(x_{0}) \subset \Omega $. Then, from the second inclusion in \eqref{half SB iteration}, we have $S_{h}(x_0) \subset \beta^i S_{t_{0}}(x_0) \subset B_{\beta^i R}(x_0)= B_{(Rt_{0}^{-\sigma})h^{\sigma}} $, where $\sigma = -\log_{2}\beta$. Since $t_{0}$ and $\beta$ depend on $n$, $\lambda$ and $\Lambda$ only, $C_{1}=Ct_{0}^{n/2-1}$ and $C_{2}=Rt_{0}^{-\sigma}$ also depend on $n$, $\lambda$ and $\Lambda$ only.

Now, for any $2^{-(i+1)}t_{0} \leq h \leq \in 2^{-i}t_{0}$, let $h'=2^{-(i+1)}t_{0}$ and $h'' = 2^{-i}t_{0}$ for convenience, then we have 
\begin{align*}
    B_{\frac{1}{2}C_{1}h}(x_{0}) \subset B_{C_{1}h'}(x_{0}) \subset S_{h'}(x_{0}) \subset S_{h}(x_{0}) \subset S_{h''} (x_{0}) \subset B_{C_{2}(h'')^{\sigma}}(x_{0}) \subset C_{C_{2}(2h)^{\sigma}}(x_{0}).
\end{align*}
Finally, redefine $C_{1} = C_{1}/2$ and $C_{2} = 2^{\sigma} C_{2}$, the proof is complete.
\end{proof}

\subsection{Linearized Monge-Amp\`ere equation}
In this subsection, we recall regularity results of the linearized Monge-Amp\`ere equation. We first review the Harnack inequality (Theorem 5, \cite{CG97}).
\begin{theorem}[Harnack inequality]\label{HI}
Assume that $\lambda \leq \det D^{2} w \leq \Lambda$ in $\Omega$ and $ w=0$ on $\partial\Omega$. Let $v\in W_{\loc}^{2,n}(\Omega)$ be a nonnegative solution to the linearized Monge-Amp\`ere equation
\begin{align*}
    L_{ w}v = 0 \quad \text{in} \quad \Omega.
\end{align*}
Then, for any $ S_{2h}(x) \subset \subset \Omega $, there exists a constant $C$, which depends on $n$, $\lambda$ and $\Lambda$ such that
\begin{align*}
    \sup_{S_{h}(x)} v \leq C \inf_{S_{h}(x)} v.
\end{align*}
\end{theorem}

Based on the Harnack inequality, Guti\'errez and Nguyen in \cite{GN15} showed the interior $W^{2,p}$ estimate of the solution $v$.

\begin{lemma}\label{interior gradient holder}
    Let \(\Omega\) be a normalized convex domain and \(f \in C(\Omega)\) with \(0 < \lambda \leq f \leq \Lambda\). Suppose \(v \in W_{\text{loc}}^{2,n}(\Omega)\) is a solution of \(L_w v = g\) in \(\Omega\), where \(w \in C(\Omega)\) is a convex function satisfying \(\det D^2 w = f\) in \(\Omega\) and \(w = 0\) on \(\partial \Omega\). Let \(\Omega' \subset \subset \Omega\), \(p > 1\), and \(\max \{n, p\} < q < \infty\). Then there exists \(C > 0\) depending only on \(p\), \(q\), \(\lambda\), \(\Lambda\), \(n\), \(\text{dist}(\Omega', \partial \Omega)\), and the modulus of continuity of \(f\) such that
\[
\| D^2 v \|_{L^p(\Omega')} \leq C \left( \| v \|_{L^\infty(\Omega)} + \| g \|_{L^q(\Omega)} \right).
\]
It is then easy to get the H\"older estimate for the gradient of $u$ by the Sobolev embedding:
\[
\| v \|_{C^{1,\alpha}(\Omega')} \leq C \left( \| v \|_{L^\infty(\Omega)} + \| g \|_{L^q(\Omega)} \right).
\]
\end{lemma}

The crucial tools in the proof of the gradient estimate are the following two lemmas, which were proposed by Guti\'errez and Nguyen in \cite{GN11}. We will use these lemmas in the proof of the main theorem. 
\begin{lemma}\label{estimate}
    Suppose $B_{1}\subset\Omega\subset B_{n}$ is a normalized convex domain. Then there exist constants $\mu_{0}>0$ and $\tau_{0}>0$ and a positive definite matrix $M=A^{T}A$ and $p\in\mathbb{R}^{n}$ satisfying
\begin{align*}
    \det M=1,\text{\quad} 0<c_{1}I\leq M\leq c_{2}I, \text{\quad}
\end{align*}
such that if $ w\in C(\bar\Omega)$ is a strictly convex function in $\Omega$ with
\begin{align*}
    \begin{cases}
         & 1-\varepsilon\leq\det D^{2} w\leq 1+\varepsilon \text{\quad in } \Omega\\
         &  w=0 \text{\quad on }\partial\Omega
    \end{cases}
\end{align*}
then for $0<\mu\leq\mu_{0}$ and $\varepsilon\leq\tau_{0}\mu^{2}$, we have
\begin{equation}\label{SB-estimate}
    B_{\sqrt{2}({1-C(\sqrt{\mu}+\frac{\sqrt{\varepsilon}}{\mu})})}(0)\subset\mu^{-\frac{1}{2}}TS_{\mu}(x_{0})\subset B_{\sqrt{2}({1+C(\sqrt{\mu}+\frac{\sqrt{\varepsilon}}{\mu})})}(0),
\end{equation}
where $T(x)=A(x-x_{0})$, $x_{0}$ is the minimum point of $ w$.
\end{lemma}

Moreover, a better estimate holds if $\Omega$ is close to $B_{\sqrt{2}}$.

\begin{lemma}\label{better estimate}
    If we replace the condition $B_{1}\subset\Omega\subset B_{n}$ by $B_{\sqrt{2}({1-\delta})}(0)\subset\Omega\subset B_{\sqrt{2}{(1+\delta)}}(0)$, where $0<\delta\leq 1/4$, then \eqref{SB-estimate} in Lemma \ref{estimate} can be improved by the following estimate:
\begin{equation}\label{better SB estimate}
    B_{\sqrt{2}({1-C(\delta\sqrt{\mu}+\frac{\sqrt{\varepsilon}}{\mu})})}(0)\subset\mu^{-\frac{1}{2}}TS_{\mu}(x_{0})\subset B_{\sqrt{2}({1+C(\delta\sqrt{\mu}+\frac{\sqrt{\varepsilon}}{\mu})})}(0).
\end{equation}
\end{lemma}

\section{Existence and Uniqueness}\label{s3}
In this section, we prove the existence and uniqueness of the problem. Let us first review the definition of the $L^{p}$-viscosity solution (see \cite{CCKS96}).
\begin{definition}
For a function $h:A\to\mathbb{R}$, where $A\subset\mathbb{R}^{n}$ is a measurable set, recall the definition
\begin{align*}
&ess\sup\limits_{A}h(x)=\inf\{M\in\mathbb{R}| h\leq M \text{ a.e.  in } A\}\\
&ess\inf\limits_{A}h(x)=\sup\{M\in\mathbb{R}| h\geq M \text{ a.e.  in } A\}.
\end{align*}
We call that a function \( u \in C(\Omega) \), where \( \Omega \subset \mathbb{R}^n \), is an open domain, is an $L^{p}$-viscosity subsolution (supersolution) to $L_{w}u=0$ if for any \( \phi \in W_{\loc}^{2,p}(\Omega) \) such that \( u - \phi \) has a local maximum (minimum) at \( x_0 \in \Omega \), we have $\lim\limits_{r\to 0} ess\sup\limits_{B_{r}(x_{0})} L_{w}\phi \geq 0\text{ }(\lim\limits_{r\to 0} ess\inf\limits_{B_{r}(x_{0})}L_{w}\phi \leq 0) $. A function \( u \) is an $L^{p}$-viscosity solution if it is both an $L^{p}$-viscosity subsolution and an $L^{p}$-viscosity supersolution.
\end{definition} 

Here we introduce $L^{p}$-viscosity solution as $D^{2}w$ is not defined pointwise. In addition, we will make use of the following Comparison Principle to show the uniqueness (general comparison principle can be found in \cite{MIL92}, section 3).

\begin{lemma}[Comparison Principle]\label{CP}
Let $\Omega \subset \mathbb{R}^n$ be an open and bounded domain, and let $w \in C(\overline{\Omega})$ be a convex function satisfying $\det D^2 w = f$ in $\Omega$ for some continuous function $0 < \lambda \leq f \leq \Lambda < \infty$. Suppose $u, v \in C(\overline{\Omega})$ are $L^{p}$-viscosity subsolution and $L^{p}$-viscosity supersolution to $L_w v = 0$ in $\Omega$ respectively. If $u \leq v$ on $\partial \Omega$, then $u \leq v$ in $\Omega$.
\end{lemma}

\begin{proof}
Let $\phi=v-u$, then $\phi$ is an $L^{p}$-viscosity supersolution to $L_{w}\phi=0$ and $\phi\geq 0$ on $\partial\Omega$. Assume that the infimum value of $\phi$ is $-M$ for some $M>0$. Since $\phi$ is continuous, the set $\{\phi=-M\}$ is closed in $\Omega$. We assume that $\{\phi=-M\}$ is not open, otherwise $\{\phi=-M\}=\Omega$, which contradicts to the boundary condition $\phi\geq 0$. For $x\in \partial\{\phi=-M\}$, select a different point $x_{0}\in \{\phi=-M\}$ (if $x\in\{\phi=-M\}$ is the only point, then select $x_{0}$ close to $x$) and define $\eta(y)=\delta|y-x_{0}|^{2}-C$. Then for $\delta>0$ small enough and an appropriate $C>0$, $\eta$ is convex and it touches $\phi$ from below at $x$. This leads to a contradiction since $\phi$ is an $L^{p}$-viscosity supersolution to $L_{w}\phi=0$.
\end{proof}

Now we give a proof of Theorem \ref{EU}.
\begin{proof}[Proof of Theorem \ref{EU}]
Define
\begin{align*}
    \Sigma = \{v\in W_{\loc}^{2,n}(\Omega)\cap C(\bar{\Omega}): v\geq \varphi, L_{ w}v\leq 0, v|_{\partial\Omega}=0\}.
\end{align*}
We want to show
\begin{align*}
    u(x)=\inf_{v\in \Sigma} v(x)
\end{align*}
is a unique $L^{n}$-viscosity solution to \eqref{LMA}.

First, it is easy to see $\Sigma$ is not empty, since the function $v(x)=a \text{ } \dist(x,\partial\Omega)^{\sigma}$ is clearly in $\Sigma$ for some appropriate $\sigma \in (0,1)$ and $a > 0$ large enough. Moreover, it is obvious that $u=0$ on $\partial\Omega$.

Next, suppose $v_{n}$ is a decreasing sequence such that $\lim \limits_{n \rightarrow \infty}v_{n}(x) = u(x)$. We define the `dropping', $V_{n}$, of $v_{n}$ in a section around a point $x \in \{u>\varphi\}$. Namely, for $S_{h,w}(x) \subset \{u>\varphi\}$ and any compact subset $K \subset S_{h,w}(x)$, there exists a $V_{n}$ satisfying
\begin{align*}
    \begin{cases}
         &L_{ w} V_{n}=0 \text{\quad in } K\\
         &V_{n} = v_{n} \text{\quad in } \bar{\Omega} \setminus K.
    \end{cases}
\end{align*}
The existence of such $V_{n}$ will be proved in the following Theorem \ref{Dropping} after normalizing $S_{h,w}(x)$. Then by comparison principle (Lemma \ref{CP}), we have $V_{n} \in \Sigma$ and $u \leq V_{n} \leq v_{n}$ in $\bar{\Omega}$, and it follows that $V_{n}(x)\to u(x)$ locally uniformly in the interior of $\{u>\varphi\}$ by Arzelà–Ascoli theorem for H\"older continuous functions (H\"older continuous is a consequence of Harnack inequality). Therefore in $\{u>\varphi\}$, $u$ is continuous and $L_{ w}u=0$.

It remains to show $u$ is continuous across the free boundary. To see this, assume that $u$ is discontinuous at $x_{0} \in \partial \{u > \varphi\}$. Since $u$ is the pointwise infimum of a decreasing sequence of continuous functions, $u$ is upper semicontinuous. Let $\{x_{n}\}_{n=1}^{\infty} \subset \{u > \varphi \}$ be a sequence of points converging to $x_{0}$, we have $\lim_{n\to\infty} u(x_{n}) + M = u(x_{0})$ for some $M>0$. For $\varepsilon>0$ small enough, choose $n$ large enough, such that $|\varphi(x_{0})-\varphi(x_{n})| < \varepsilon$, then choose $m$ large enough such that $V_{m}(x_{0})-V_{m}(x_{n}) \geq M/2$, $V_{m}(x_{0})-u(x_{0}) < \varepsilon$ and $u(x_{n})-V_{m}(x_{n}) < \varepsilon$. Then it follows that
\begin{align*}
\varphi(x_{n}) > \varphi(x_{0}) - \varepsilon > u(x_{0}) - M/4 
> V_{m}(x_{0}) - \varepsilon - M/4 > V_{m}(x_{n}).
\end{align*}
This leads to a contradiction since $V_m \geq \varphi$.

Finally, we will show the uniqueness of the solution. Assume that $u_{1}$ and $u_{2}$ are two solutions to \eqref{main}, then consider the set $\{u_{1} > u_{2}\}$, we have $u_{1} > \varphi$. Therefore,
\begin{align*}
    \begin{cases}
        &L_{w}u_{1} \geq L_{w}u_{2} \quad\text{  in  }\quad \{u_{1} > u_{2}\} \\
        &u_{1}=u_{2} \quad\text{  on  }\quad \partial\{u_{1} > u_{2}\}.
    \end{cases}
\end{align*}
It follows that $u_{1} \leq u_{2}$ by the comparison principle, and this contradicts the assumption $\{u_{1} > u_{2}\}$.
\end{proof}

Now, let us return to prove the existence of $V_n$ in the interior of a section. Without loss of generality, we assume that the section has already been normalized in the following theorem.

\begin{theorem}\label{Dropping}
    Let $B_{3/2} \subset D \subset B_{n}$ be an open bounded convex domain, $w\in C(\bar \Omega)$ is a convex function satisfying 
    \begin{align*}
        \begin{cases}
            &1-\varepsilon \leq \det D^{2}w=f \leq 1+\varepsilon\\
            &f\in C(\bar{D}), \text{ and }w=0 \text{ on }\partial D.
        \end{cases}
    \end{align*}
    If $v\in W_{\loc}^{2,n}(D)\cap C(\bar{D})$ satisfies 
    \begin{align*}
        \begin{cases}
            & L_{w}v\leq 0 \text{ in }D\\
            & v = g \text{ on }\partial D,
        \end{cases}
    \end{align*}
    where $g \in C(\partial\Omega)$, then there exists a $V \in W^{2,n}(B_{1})\cap C(\bar{B_{1}})$ that satisfies
    \begin{equation}\label{dropping equation}
        \begin{cases}
            &L_{ w}V=0 \text{\quad in }B_{1}\\
            &V=v \text{\quad on } \partial B_{1}.
        \end{cases}
    \end{equation}
\end{theorem}

\begin{proof}
Let $f_{k}$ be a family of smooth functions on $\bar{D}$ that converges to $f$ uniformly in $\bar{D}$, and let $p>n$ and $q > (n-1)p$. By the stability of the Monge-Amp\`ere operator, there exists a sequence of $w_{k}$ such that $w_{k}\to w$ uniformly, where $w_{k}$ is the solution to the Monge-Amp\`ere equation $\det D^{2}w_{k}=f_{k}$. In addition, by the interior $W^{2,q}$ estimate of the Monge-Amp\`ere equation (see Theorem 1 in \cite{C90}), $(w_{k})_{k=1}^{\infty}$ is a bounded sequence in $W^{2,q}(B_{1})$, and thus $w_{k}$ converges to $w$ weakly in $W^{2,q}$. Note that $f_{k}$ is a smooth function, then the linearized Monge-Amp\`ere operator $L_{w_{k}}$ is uniformly elliptic. Therefore, it follows that there exists a solution $V_{k}$ to the problem
\begin{align*}
    \begin{cases}
            & L_{w_{k}}V_{k}=0 \text{ in } D\\
            & V_{k} = v \text{ on } \partial D.
    \end{cases}
\end{align*}
By Lemma \ref{interior gradient holder}, we have $||D^{2}V_{k}||_{L^{s}(B_{5/4})} \leq C||v||_{L^{\infty}(\partial D)}$ for any $s<\infty$, and $C$ is independent to $k$. Then, there exists a subsequence $V_{k_{j}}$ that converges to $V$ weakly in $W^{2,s}(B_{5/4})$ (we will still use $V_{k}$ to denote this subsequence for convenience), and by Sovolev embedding, $V_{k}$ converges to $V$ in $C^{1,\alpha}(\bar{B_{1}})$. This implies that $V\in C^{1,\alpha}(\bar{B_{1}})$. 
    
Now, it remains to show $W_{ij} V_{ij}=0$ almost everywhere. For any test function $\xi \in C_{c}(B_{1})$,
\begin{align*}
    \int_{B_{1}} W_{ij}V_{ij}\xi &= \int_{B_{1}} \big(W_{ij}V_{ij}-(W_{k})_{ij}(V_{k})_{ij}\big)\xi \\
    & = \int_{B_{1}} \big(W_{ij}V_{ij}-(W_{k})_{ij}V_{ij})\big)\xi + \int_{B_{1}}\big((W_{k})_{ij}V_{ij}-(W_{k})_{ij}(V_{k})_{ij}\big)\xi\\
    & = \int_{B_{1}} \big(W_{ij}-(W_{k})_{ij}\big)V_{ij}\xi + \int_{B_{1}}(W_{k})_{ij}\big(V_{ij}-(V_{k})_{ij}\big)\xi\\
    & = I + II.
\end{align*}

The first integral approaches to $0$ since $D^{2}w_{k}$ converges to $D^{2}w$ weakly in $L^{q}$, it follows that $(W_{k})_{ij}$ converges to $W_{ij}$ weakly in $L^{\frac{q}{n-1}}$. Moreover $V_{ij}\xi\in L^{\frac{q}{q-(n-1)}}(B_{1})$ as $V_{ij}\in W^{2,s}(B_{5/4})$ and $\xi \in C_{c}(B_{1})$. By choosing $s$ large enough, the weak convergence implies that $I$ converges to $0$.

The second integral also converges to $0$, since the linearized Monge-Amp\`ere operator is divergence-free, namely $\sum_{i=1}^{n}\partial_{i}(W_{k})_{ij}=0$. Then by integration by parts, 
$$
\int_{B_{1}}(W_{k})_{ij}\big(V_{ij}-(V_{k})_{ij}\big)\xi=\int_{B_{1}}(W_{k})_{ij}(V-V_{k})_{i}\xi_{j}.
$$
Since $V_{k}$ converges to $V$ in $C^{1,\alpha}(\bar{B_{1}})$ and $W_{k}$ is uniformly elliptic, $II$ converges to $0$ as $k\to\infty$. Since $\xi$ is arbitrary, we conclude that $V \in W^{2,n}(B_{1})\cap C(\bar{B_{1}})$ is a solution to \eqref{dropping equation}.
\end{proof}

\section{Regularity of the solution}\label{s4}
To make the proof clear, we divide it into several lemmas. By John's lemma, it suffices to consider the case where $\Omega$ is normalized. Lemma \ref{u-l} shows that $u$ cannot `grow too much' away from the free boundary. Then we normalize a section around $x_{0} \in \partial \{u>\varphi\}$, and we have the continuity of $Du^{*}$ in Lemma \ref{gradient continuous} via the iteration method, where $u^{*}$ is obtained from $u$ via an affine transformation. Based on the result, in Lemma \ref{gradient growth} we further obtain the growth of $Du^{*}$, which can lead us to show the H\"{o}lder estimate for $Du^{*}$, and hence the $C^{1,\gamma}$ regularity of $u$. In the following context, the constant $C$ may vary if it is universal.

\begin{lemma}\label{u-l}
Suppose that \eqref{condition} holds and $\Omega$ is normalized, $u \in W_{\loc}^{2,n}(\Omega) \cap C(\bar{\Omega})$ is a solution to $\eqref{LMA}$. Let $x_{0}\in\partial\{u>\varphi\}$, $S_{h}(x_{0})\subset\Omega$ be a section of $ w$. If
\begin{align*}
\sup\limits_{S_{h}(x_{0})}|\varphi(x)-l_{x_{0}}(x)|\leq\kappa,
\end{align*}
where $l_{x_{0}}(x)=\varphi(x_{0})+D\varphi(x_{0})(x-x_{0})$, then
\begin{align*}
    \sup\limits_{S_{h/2}(x_{0})}|u(x)-l_{x_{0}}(x)|\leq C\kappa.
\end{align*}
\end{lemma}

\begin{proof}
Since $u\geq\varphi$, we have
\begin{align*}
    u(x)-l_{x_{0}}(x)\geq\varphi(x)-l_{x_{0}}(x)\geq-\kappa.
\end{align*}
Therefore, we only need to show
\begin{align*}
    u(x)-l_{x_{0}}(x)\leq C\kappa.
\end{align*}

Let $v(x)=u(x)-l_{x_{0}}(x)+\kappa$, then $v\geq0$ and $L_{ w}u=L_{ w}v$ in $S_{h}(x_{0})$. Let $V$ be a solution to
\begin{align*}
    \begin{cases}
        L_{ w}V=0 \text{\quad in \quad} S_{h}(x_{0})\\
        V=v \text{\quad on \quad} \partial S_{h}(x_{0}).
    \end{cases}
\end{align*}

We claim that $0\leq V\leq v\leq V+2\kappa$ in $S_{h}(x_{0})$. Note that the existence of $V$ is followed by Theorem \ref{Dropping} in $S_{2h}(x_{0})$. $V\geq 0$ in $S_{h}(x_{0})$ follows by the maximum principle. By the comparison principle, $V\leq v$ since $L_{ w}v\leq 0$ in $\Omega$. Therefore, it remains to show $v\leq V+2\kappa$. Note that on $\partial S_{h}(x_{0})$, $V=v\leq V+2\kappa$, and in $S_{h}(x_{0})\cap\{u=\varphi\}$, $v(x)=\varphi(x)-l_{x_{0}}(x)+\kappa\leq 2\kappa\leq V+2\kappa$. Therefore, 
$$v\leq V+2\kappa \text{ on } \partial (S_{h}(x_{0})\cap\{u>\varphi\}),
$$ 
where $\partial (S_{h}(x_{0})\cap\{u>\varphi\})=(\partial S_{h}(x_{0})\cap\{u>\varphi\})\cap(S_{h}(x_{0})\cap\partial\{u>\varphi\})$. In addition,
\begin{align*}
    \begin{cases}
        &L_{ w}V=L_{ w}v=0 \text{\quad in \quad}S_{h}(x_{0})\cap\{u>\varphi\}\\
        &v\leq V+2\kappa \text{\quad on \quad} \partial (S_{h}(x_{0})\cap\{u>\varphi\}),
    \end{cases}
\end{align*}
then by the comparison principle, $v\leq V+2\kappa$ in $S_{h}(x_{0})\cap\{u>\varphi\}$, which proves the claim.

Since $L_{ w}V=0$ in $S_{h}(x_{0})$, by the Harnack inequality (Theorem \ref{HI}),
\begin{align*}
    \sup\limits_{S_{h/2}(x_{0})}V\leq C\inf\limits_{S_{h/2}(x_{0})}V\leq CV(x_{0}) \leq  Cv(x_{0})\leq C\kappa.
\end{align*}
Therefore, $v(x)\leq C\kappa$ in $S_{h/2}(x_{0})$, and hence
\begin{align*}
    \sup\limits_{S_{h/2}(x_{0})}|u(x)-l_{x_{0}}(x)|\leq C\kappa.
\end{align*}
\end{proof}

In the following two lemmas, we will consider the section around $x_{0}$ after normalizing the section $S_{h}(x_{0})$. Specifically, let $T(x)=Ax+b$ be the affine transformation that normalizes $S_{h}(x_{0})$, define
\begin{equation}\label{neighbourhood free boundary}
        \begin{cases}
            &y_{0}=Tx_{0}\\
            & w^{*}(y)=K\big( w(T^{-1}y)-l_{ w,x_{0}}(T^{-1}y) - h\big)\\
            &u^{*}(y)=u(T^{-1}y)\\
            &\varphi^{*}(y)=\varphi(T^{-1}y)\\
            &\Omega_{0}^{*} = T S_{h}(x_{0}),
        \end{cases}
    \end{equation}
where $l_{ w,x_{0}}$ is the linear part of $w$ at $x_{0}$, $K=\frac{|\det A|^{\frac{2}{n}}}{\det D^{2} w(x_{0}))^{\frac{1}{n}}}$. It is clear that $u^{*} \geq \varphi^{*}$ in $\Omega_{0}^{*}$. By some direct calculations, we have
\begin{align*}
    \begin{cases}
        &D^{2} w^{*}(y)= K (A^{-1})^{T} D^{2}  w(T^{-1}y) A^{-1}\\
        &\det D^{2} w^{*}(y) = \frac{\det D^{2} w(T^{-1}y)}{\det D^{2} w(x_{0})}=\frac{f(T^{-1}y)}{f(x_{0})}\\
        &  W^{*}(y)= \frac{K}{\det D^{2} w(x_{0})} A  W A^T\\
        & D^{2}u^{*}(y) = (A^{-1})^{T} D^{2}u(T^{-1}y) A^{-1}.
    \end{cases}
\end{align*}
It follows that $L_{ w^{*}}u^{*} = 0$ in $\{u^{*} > \varphi^{*}\}$ and $u^{*}$ is a solution to $\varphi^{*}$-obstacle problem in $\Omega_{0}^{*}$. 

\begin{lemma}\label{gradient continuous}
Let $u \in W_{\loc}^{2,n}(\Omega) \cap C(\bar{\Omega})$ be a solution to \eqref{LMA}, $x_{0}\in\partial\{u>\varphi\}$ be a free boundary point, $S_{h}(x_{0})\subset\Omega$ be a section of $ w$ at $x_{0}$. Define $ w^{*}$, $u^{*}$ and $\varphi^{*}$ as above in \eqref{neighbourhood free boundary} then $Du^{*}$ is continuous at $y_{0}$.
\end{lemma}

\begin{proof}
Since $\Omega_{0}^{*}$ is normalized, and $y_{0}$ is the minimum point, we will apply Lemma \ref{estimate}. Fixed $h_{0}\leq\mu_{0}$, where $\mu_{0}$ is defined in Lemma \ref{estimate}, we can find a $h$ small enough such that $1 - \varepsilon \leq \det D^{2} w^{*} \leq 1 + \varepsilon$ in $\Omega_{0}^{*}$, where $\varepsilon\leq \tau_{0}\mu_{0}^{2}$ and $\tau_{0}$ is defined in Lemma \ref{estimate}. We claim that for each $k>0$, there exists a positive definite matrix $A_{k}$ and a $\delta_{k}>0$, such that
\begin{align}
    &B_{\sqrt{2}({1-\delta_{k}})}(0)\subset h_{0}^{-\frac{k}{2}}A_{k}S_{h_{0}^{k}, w^{*}}(x_{0})\subset B_{\sqrt{2}({1+\delta_{k}})}(0)\label{1}\\
    &\Big(\sqrt{c_{1}}\prod\limits_{i=1}^{k-1}\sqrt{(1-C\delta_{i})} \Big)I\leq A_{k}\leq\Big(\sqrt{c_{2}}\prod\limits_{i=1}^{k-1}\sqrt{(1+C\delta_{i})}\Big)I\label{2},
\end{align}
where $\delta_{0}=0$ and $\delta_{k}=C(\delta_{k-1}\sqrt{h_{0}}+\frac{\sqrt{\varepsilon}}{h_{0}})$. \\
\\
We will show by induction.\\
\\
When $k=1$, \eqref{1} and \eqref{2} hold obviously by Lemma \ref{estimate}.\\
\\
When $k=2$, let $\Omega_{1}^{*}=h_{0}^{-\frac{1}{2}}A_{1}S_{h_{0}, w^{*}}$, and
\begin{align*}
    \eta^{*}(y)=\frac{1}{h_{0}}\big( w^{*}(h_{0}^{\frac{1}{2}}A_{1}^{-1}y- w^{*}(0)-h_{0})\big), \text{\quad} y\in\Omega_{1}^{*}.
\end{align*}
By Lemma \ref{better estimate}, $B_{\sqrt{2}({1-\delta_{1}})}(0)\subset\Omega_{1}^{*}\subset B_{\sqrt{2}({1-\delta_{1}})}(0)$ and hence there exists a positive definite matrix $A$ and $\delta_{2}=C(\delta_{1}\sqrt{h_{0}}+\frac{\sqrt{\varepsilon}}{h_{0}})$, such that
\begin{align}
    B_{\sqrt{2}({1-\delta_{2}})}(0)\subset h_{0}^{-\frac{1}{2}}A S_{h_{0},\eta^{*}}(x_{0})\subset B_{\sqrt{2}({1-\delta_{2}})}(0)\label{3}.
\end{align}
Note that $S_{h_{0},\eta^{*}}(x_{0})=h_{0}^{-\frac{1}{2}}A_{1}S_{h_{0}^{2}}(x_{0})$, define $A_{2}=AA_{1}$, then \eqref{1} follows by \eqref{3}. For \eqref{2}, by Lemma \ref{better estimate}, we have
\begin{align*}
     (\sqrt{1-C\delta_{1}})I\leq A\leq(\sqrt{1+C\delta_{1}})I.
\end{align*}
Therefore,
\begin{align*}
    (\sqrt{c_{1}(1-C\delta_{1})})I\leq A_{2}\leq(\sqrt{c_{2}(1+C\delta_{1})})I,
\end{align*}
which completes the proof of the claim when $k=2$.\\
\\
Suppose the claim holds for all $k\leq i$, we want to prove that it still holds for $k=i+1$. Let $\Omega_{i}^{*}=h_{0}^{-\frac{i}{2}}A_{i}S_{h_{0}, w^{*}}$, and
\begin{align*}
    \eta^{*}(y)=\frac{1}{h_{0}}\big(\eta(h_{0}^{\frac{1}{2}}A_{i}^{-1}y-\eta(0)-h_{0})\big), \text{\quad} y\in\Omega_{i}^{*}.
\end{align*}
By Lemma \ref{better estimate}, $B_{\sqrt{2}({1-\delta_{i}})}(0)\subset\Omega_{i}^{*}\subset B_{\sqrt{2}({1-\delta_{i}})}(0)$ and hence there exists a positive definite matrix $A$ and $\delta_{i+1}=C(\delta_{i}\sqrt{h_{0}}+\frac{\sqrt{\varepsilon}}{h_{0}})$, such that
\begin{align}
    B_{\sqrt{2}({1-\delta_{i+1}})}(0)\subset h^{-\frac{1}{2}}A S_{h,\eta^{*}}(x_{0})\subset B_{\sqrt{2}({1-\delta_{i+1}})}(0)\label{4}.
\end{align}
Note that $S_{h_{0},\eta^{*}}(x_{0})=h_{0}^{-\frac{1}{2}}A_{i}S_{h_{0}^{i+1}}(x_{0})$, define $A_{i+1}=AA_{i}$, then \eqref{1} follows by \eqref{4}, and we have
\begin{align*}
     (\sqrt{1-C\delta_{i}})I\leq A\leq(\sqrt{1+C\delta_{i}})I.
\end{align*}
Therefore,
\begin{equation}\label{iteration estimate}
    (\sqrt{c_{1}}\prod\limits_{j=1}^{i}\sqrt{(1-C\delta_{j})})I\leq A_{i+1} \leq (\sqrt{c_{2}}\prod\limits_{j=1}^{i}\sqrt{(1+C\delta_{j})})I,
\end{equation}
which completes the proof of the claim.

Note that $\delta_{k}=C(\delta_{k-1}\sqrt{h_{0}}+\frac{\sqrt{\varepsilon}}{h_{0}})$ is a decreasing sequence, by induction,
\begin{align*}
    \delta_{k} = (C\sqrt{h_{0}})^{k}+\frac{C\sqrt{\varepsilon}}{h_{0}}\sum\limits_{i=0}^{k-1}(C\sqrt{h_{0}})^{i} \leq (C\sqrt{h_{0}})^{k}+\frac{2C\sqrt{\varepsilon}}{h_{0}}.
\end{align*}
Thus, we have
\begin{align*}
    \prod\limits_{i=1}^{k-1}(1+C\delta_{i})  &\leq \exp\Big(\sum_{i=1}^{k-1}\ C\delta_{i}\Big) \leq \exp \Big(C\sum_{i=1}^{k-1}(C\sqrt{h_{0}})^{k}+\frac{2C(k-1)\sqrt{\varepsilon}}{h_{0}}\Big) \\
    &\leq \exp \Big(C\sum_{i=1}^{k-1}(C\sqrt{h_{0}})^{k}\Big)\exp\Big(\frac{2C(k-1)\sqrt{\varepsilon}}{h_{0}}\Big)\\ 
    &\leq C \exp\Big({\frac{2C\sqrt{\varepsilon}k}{h_{0}}}\Big),
\end{align*}
where the first inequality follows the fact that $1+x \leq e^{x}$. If we set \begin{equation}\label{varepsilon and theta}
    \sqrt{\varepsilon}\leq \theta\cdot\frac{h_{0} \ln h_{0}^{-1}}{2C}, 
\end{equation}
where the constant $\theta>0$ will be determined later, then the second inequality of \eqref{iteration estimate} will lead to
\begin{align}
    A_{k} \leq (Ch_{0}^{-\frac{\theta}{2}k})I.\label{5}
\end{align}
Similarly, since $1-x \geq e^{-2x}$ for $x>0$ small, then for the first inequality of \eqref{iteration estimate}, we have, 
\begin{align*}
    \prod\limits_{i=1}^{k-1}(1-C\delta_{i})& \geq \exp\Big(\sum\limits_{i=1}^{k-1} -2C\delta_{i} \Big)\\
    &\geq \exp\Big(-2\sum\limits_{i=1}^{k-1}(C\sqrt{h_{0}})^{k}-\frac{4(k-1)C\sqrt{\varepsilon}}{h_{0}}\Big)\\
    &\geq  C \exp\Big( -\frac{4kC\sqrt{\varepsilon}}{h_{0}}\Big). 
\end{align*}
Since $\sqrt{\varepsilon}\leq \theta\cdot\frac{h_{0} \ln h_{0}^{-1}}{2C}$, 
\begin{equation}
    A_{k}\geq (Ch_{0}^{\theta k})I.\label{6}
\end{equation}
It follows that
\begin{align}
    (Ch_{0}^{\frac{\theta}{2}k})I\leq A_{k}^{-1} \leq (Ch_{0}^{-\theta k})I \label{8}.
\end{align}
Therefore, by \eqref{1}, if $y\in S_{h_{0}^{k},w^{*}}(y_{0})\setminus S_{h_{0}^{k+1},w^{*}}(y_{0})$ we can obtain
\begin{align}
    \sqrt{2}Ch_{0}^{\frac{(1+\theta)}{2}(k+1)} \leq |y-y_{0}| \leq \sqrt{2} Ch_{0}^{(\frac{1}{2}-\theta) k}.\label{9}
\end{align}

Now we are ready to show the continuity of $Du^{*}$ at $y_{0}$. Let $e\in\mathbb{S}^{n-1}$ be a unit vector and $y=y_{0}+t\cdot e$, and we consider the directional derivative $D_{e}u^{*}$ at $y_{0}$. By definition, $D_{e}u^{*}(y_{0})=\lim\limits_{t\to 0}\frac{u^{*}(y)-u^{*}(y_{0})}{t}$. If $y\in\{u^{*}=\varphi^{*}\}$, it is easy to check that $D_{e}u^{*}(y_{0})=D_{e}\varphi^{*}(y_{0})$ since $u(y_{0})=\varphi(y_{0})$ on the free boundary. Therefore, we only need to show the case when $y\in\{u^{*} > \varphi^{*}\}$. Let $y\in S_{h_{0}^{k},w^{*}}(y_{0})\setminus S_{h_{0}^{k+1},w^{*}}(y_{0})$, note that
\begin{align}
    D_{e}u^{*}(y_{0})=\lim\limits_{t\to 0}\frac{u^{*}(y)-u^{*}(y_{0})}{t}=\lim\limits_{t\to 0}\frac{u^{*}(y)-l_{0}^{*}(y)}{t}+D_{e}\varphi^{*}(y_{0})\label{10},
\end{align}
where $l^{*}_{0}(y)$ is the linear part of $\varphi^{*}$ at $y_{0}$. By \eqref{9}, we have
\begin{align}
    \sup\limits_{y\in S_{h_{0}^{k-1},\varphi^{*}}(y_{0})}|\varphi^{*}(y)-l^{*}_{0}(y)|\leq \sup\limits_{y\in S_{h_{0}^{k-1},\varphi^{*}}(y_{0})}||D^{2}\varphi^{*}||_{L^{\infty}(S_{h_{0}, \varphi^{*}}(y_{0}))}\cdot|y-y_{0}|^{2}
    \leq C h_{0}^{(1-2\theta)(k-1)},\label{11}
\end{align}
where we can choose $h\leq 1/2$. Therefore, by Lemma \ref{u-l}
\begin{align*}
    |D_{e}u^{*}(y_{0})-D_{e}\varphi^{*}(y_{0})|&=\lim\limits_{t\to 0}\frac{u^{*}(x)-l^{*}_{0}(x)}{t}\\
    &\leq\lim\limits_{t\to 0}\sup\limits_{x\in S_{{h_{0}^{k},\eta^{*}}}(y_{0})}\frac{|u^{*}(x)-l^{*}_{0}(x)|}{|y-y_{0}|}\\
    &\leq\lim\limits_{t\to 0}\sup\limits_{y\in S_{{h_{0}^{k-1},\eta^{*}}}(y_{0})}C\frac{|\varphi^{*}(y)-l^{*}_{0}(y)|}{|y-y_{0}|}\\
    &\leq\lim\limits_{k\to\infty}C\frac{h_{0}^{(1-2\theta)(k-1)}}{h_{0}^{\frac{(1+\theta)}{2}(k+1)}}\\
    &\leq\lim\limits_{k\to\infty}Ch_{0}^{\frac{\theta-3}{2}}\cdot h_{0}^{\frac{1-5\theta}{2}k}.
\end{align*}
When $\theta\leq\frac{1}{5}$, the limit goes to $0$ when $k\to\infty$, which proves that $Du^{*}$ is continuous at $y_{0}$.
\end{proof}

\begin{remark}
    In the proof of Lemma \ref{gradient continuous}, $||D^{2} \varphi^{*}||_{L^{\infty}(\Omega_{0}^{*})}$ is finite once we fix the height of the section $S_{h,w}(x_{0})$.
\end{remark}

\begin{lemma}\label{gradient growth}
    Let $x_{0}\in\partial\{u>\varphi\}$, and $S_{h}(x_{0})\subset\Omega$ be a section of $ w$ at $x_{0}$. If $T$ is an affine transformation that normalizes $S_{h}(x_{0})$, Define $ w^{*}$, $u^{*}$ and $\varphi^{*}$ as above in \eqref{neighbourhood free boundary} and suppose that $1 - \varepsilon \leq \det D^{2} w^{*} \leq 1 + \varepsilon$ in $\Omega_{0}^{*}$ for $h$ small enough, then $Du^{*}$ grows at most as $r^{\alpha}$ away from $y_{0}$. In other words, $|Du^{*}(y)-Du^{*}(y_{0})|\leq Cr^{\alpha}$ for some $\alpha\in(0,1)$, where $y\in\{u^{*}>y^{*}\}$ is a point near $y_{0}$ and $r=\dist(y,y_{0})$.
\end{lemma}

\begin{proof}
Pick an $y\in\{u^{*} > \varphi ^{*}\}$ such that $d(y,\partial\{u^{*} > \varphi^{*}\})=d(y,y_{0})=r$. Then $B_{r}(y)\subset\{u>\varphi\}$, and $B_{r}(y)\subset B_{2r}(y_{0})$. 

We claim that
\begin{align}
    B_{(\sqrt{2}C)^{-1} h_{0}^{\frac{(1+\theta)k}{2}}}(y_{0})\subset S_{h_{0}^{k},w^{*}}(y_{0})\label{12}.
\end{align}
Indeed, if $y\in B_{(\sqrt{2}C)^{-1} h_{0}^{\frac{1+\theta}{2}k}}(y_{0})$, then by \eqref{5} 
\begin{align*}
    h_{0}^{-\frac{k}{2}}||A_{k}||\cdot|y-y_{0}|\leq C h_{0}^{-\frac{k}{2}}\cdot h_{0}^{-\frac{\theta}{2}k}\cdot(\sqrt{2}C)^{-1} h_{0}^{\frac{(1+\theta)k}{2}}\leq\frac{1}{\sqrt{2}}.
\end{align*}
Hence, by \eqref{2}, $y\in S_{h_{0}^{k},w^{*}}(y_{0})$.

Let $k$ be the integer such that
\begin{align*}
B_{(\sqrt{2}C)^{-1} h_{0}^{\frac{(1+\theta)(k+1)}{2}}}(y_{0}) \subset B_{2r}(y_{0}) \subset B_{(\sqrt{2}C)^{-1} h_{0}^{\frac{(1+\theta)k}{2}}}(y_{0}) \subset S_{h_{0}^{k},w^{*}}(y_{0}).
\end{align*}
Define $v(y)=u^{*}(y)-l^{*}_{0}(y)$. By Lemma $\ref{u-l}$ and \eqref{11}, $v(y)\leq C h_{0}^{(1-2\theta)(k-1)}$ for all $y\in S_{h_{0}^{k},w^{*}}(y_{0})$. Since $\det D^{2} w^{*}(y) = \frac{\det D^{2} w(T^{-1}y)}{\det D^{2} w(x_{0})}=\frac{f(T^{-1}y)}{f(x_{0})}$ is continuous by definition, by applying Lemma \ref{interior gradient holder}, we have
\begin{align*}
    ||Dv||_{L^{\infty}\big(B_{\frac{r}{2}}(y)\big)}&\leq \frac{C}{r}||v||_{L^{\infty}(B_{r}(y))}\leq\frac{C}{r}||v||_{L^{\infty}\big(S_{h_{0}^{k},w^{*}}(y_{0})\big)}\\
    &\leq\frac{2C h_{0}^{(1-2\theta)(k-1)}}{(\sqrt{2}C)^{-1} h_{0}^{\frac{(1+\theta)(k+1)}{2}}}\\
    &\leq C h_{0}^{\frac{\theta-3}{2}}\cdot h_{0}^{\frac{(1-5\theta)k}{2}}.
\end{align*}
In particular, $|Dv(y)| \leq Ch_{0}^{\frac{\theta - 3}{2}} \cdot h_{0}^{\frac{(1-5 \theta) k}{2}}$. Therefore,
\begin{align*}
    \frac{Dv(y)-Dv(y_{0})}{r^{\alpha}}  \leq \frac{C h_{0}^{\frac{\theta-3}{2}}\cdot h_{0}^{\frac{(1-5\theta)k}{2}}}{\big((\sqrt{2}C)^{-1} h_{0}^{\frac{(1+\theta)(k+1)}{2}}\big)^{\alpha}}  \leq C h_{0}^{\frac{(1-\alpha)\theta-3-\alpha}{2}}\cdot h_{0}^{\frac{k}{2}(1-5\theta-\alpha-\alpha \theta)}.
\end{align*}
If we set $\alpha=\frac{1-5\theta}{1+\theta}$, then $C h_{0}^{\frac{(1-\alpha)\theta-3-\alpha}{2}}$ is a constant once $h_{0}$ and $\theta$ are small and fixed, and hence the proof is complete.
\end{proof}

Now, we are ready to prove the main theorem of the paper.

\begin{proof}[Proof of Theorem \ref{main}]
    As stated before, the only problem is the regularity of $u$ when it crosses the free boundary. Since $u$ is continuous and $u>\varphi$ on the boundary, $\{u=\varphi\}\subset\subset\Omega$. By Lemma \ref{section and ball}, we can find a $h_{0}>0$ such that $S_{h_{0}}(x_{0})\subset\subset\Omega$ for $h_{0}\leq t_{0}$ and $x_{0}\in\partial\{u=\varphi\}$ and
    \begin{align}
        B_{C_{1}h_{0}}(x_{0})\subset S_{h_{0}}(x_{0})\subset B_{C_{2}h_{0}^{\sigma}}(x_{0}).\label{section on the FB}
    \end{align} 
    Let $T(x)=A(x)+b$ be the affine transformation which normalized $S_{h_{0}}(x_{0})$, then
    \begin{align}
        B_{1}(0)\subset TS_{h_{0}}(x_{0})\subset B_{n}(0).\label{normalization}
    \end{align}
    By \eqref{section on the FB} and \eqref{normalization}, we can get
    \begin{equation*}
        \begin{cases}
            TB_{C_{1}h_{0}}(x_{0})\subset B_{n}(0)\\
            B_{1}(0)\subset TB_{C_{2}h_{0}^{\sigma}}(x_{0}).
        \end{cases}
    \end{equation*}
    It follows that the matrix $A$ satisfies
    \begin{equation}\label{A bound}
        (ch_{0}^{-\sigma})I\leq A\leq (Ch_{0}^{-1})I,
    \end{equation}
    which is independent to $x_{0}$.
    
    Next, choose $h_{0}$ small enough such that for $x\in S_{h_{0}}(x_{0})$, we have $|D^{2} w(x_{0})-D^{2} w(x)|\leq\varepsilon_{0}$, where $\varepsilon_{0}$ will be determined later. Define $ w^{*}$, $ u^{*} $ and $\varphi^{*}$ as in \eqref{neighbourhood free boundary}, and $K=\frac{|\det A|^{\frac{2}{n}}}{g(x_{0})^{\frac{1}{n}}}$. Then we can easily check that
    \begin{align*}
        &1-\frac{\varepsilon_{0}}{\lambda}\leq \det D^{2} w^{*}\leq1+\frac{\varepsilon_{0}}{\lambda}.
    \end{align*}
Choose $\varepsilon_{0}$ small enough, that is $h_{0}$ small enough, such that \eqref{5} and \eqref{6} hold. Then in the normalized section $TS_{h_{0}}(x_{0})$, $u^{*}$ is the solution to the $\varphi^{*}$-obstacle problem. Now, let $x_{1}$ and $x_{2}$ be two points near the free boundary, then consider $y_{1}=Tx_{1}$ and $y_{2}=Tx_{2}$. If $y_{1}$ and $y_{2}$ are both in the contact set $\{u^{*}=\varphi^{*}\}$, then the regularity of $u^{*}$ is dominated by $\varphi^{*}$. Therefore, without loss of generality, we assume that $y_{1}\in\{u^{*}>\varphi^{*}\}$. Let $d_{i}=\dist\big(y_{i},\partial\{u^{*}>\varphi^{*}\}\big)=\dist(y_{i}, y_{i}')$, where $y_{i}\in\partial\{u>\varphi\}$, $i=1,2$. \\
\\
\textbf{Case 1:} $\dist(y_{1},y_{2})\leq\frac{1}{2}\max\{d_{1},d_{2}\}$. By Lemma \ref{interior gradient holder}, the gradient of $u$ is H\"older continuous. Therefore, for the homogeneous case in $\{u^{*}>\varphi^{*}\}$, $|Du^{*}(y_{1})-Du^{*}(y_{2})|\leq C|y_{1}-y_{2}|^{\beta}$ for any $0<\beta<1$.\\
\\
\textbf{Case 2:} $\dist(y_{1},y_{2})\geq\frac{1}{2}\max\{d_{1},d_{2}\}$. By Lemma \ref{gradient growth}, we have $|Du^{*}(y_{i})|\leq Cd_{i}^{\alpha}$, $i=1,2$. Since
\begin{align*}
    |y_{1}'-y_{2}'|^{\alpha}  &\leq|y_{1}'-y_{1}|^{\alpha}  +  |y_{1}-y_{2}|^{\alpha} + |y_{2}-y_{2}'|^{\alpha}\\
    &\leq (2|y_{1}-y_{2}|)^{\alpha}  +  |y_{1}-y_{2}|^{\alpha}  +  (2|y_{1}-y_{2}|)^{\alpha}\\
    &\text{\quad(since $\dist(y_{1},y_{2})\geq\frac{1}{2}\max\{d_{1},d_{2}\}$)}\\
    &\leq C|y_{1}-y_{2}|^{\alpha},
    \end{align*}
thus we have
    \begin{align*}
        |Du^{*}(y_{1})-Du^{*}(y_{2})|&\leq |Du^{*}(y_{1})-Du^{*}(y_{1}')|  +  |Du^{*}(y_{1}')-Du^{*}(y_{2}')|  +  |Du^{*}(y_{2}')-Du^{*}(y_{2})|\\
        &\leq Cd_{1}^{\alpha}  +  C|y_{1}'-y_{2}'|^{\alpha}  +  Cd_{2}^{\alpha}\\
        &\leq C|y_{1}-y_{2}|^{\alpha}. 
    \end{align*}
Choosing $\gamma=\min\{\alpha,\beta\}$, we show that $u^{*}$ is $C_{\loc}^{1,\gamma}(\Omega)$, and it follows that $u$ is $C_{\loc}^{1,\gamma}(\Omega)$ since the matrix $A$ is bounded from \eqref{A bound}.

Finally, we claim that $\gamma\in(0,1)$ can be any number. Indeed, $\alpha=\frac{1-5\theta}{1+\theta}$ depends on the choice of $\theta$ from the proof of Lemma \ref{gradient growth}. If we choose $h_{0}$ small enough, the perturbation of $D^{2}w^{*}$, $\varepsilon_{0}/\lambda$, can be sufficiently small, and it follows that $\theta$ can be sufficiently small from \eqref{varepsilon and theta}. Therefore, $\alpha$ can be any value in $(0,1)$. Note that $\beta \in (0,1)$ can also be any number from \textbf{Case 1}, then we conclude that $\gamma = \min\{ \alpha, \beta \}$ can be any value in $(0,1)$.
\end{proof}

\begin{remark}
At this stage, a natural question is: can we get $C_{\loc}^{1,1}(\Omega)$ regularity of the solution? Under assumption \eqref{condition}, the answer is negative, since $w$ has interior $W^{2,p}$ regularity by \cite{C90}, where $p<\infty$. If one wants to get $C_{\loc}^{1,1}(\Omega)$ regularity for $u$, then $w$ needs to be a $W^{2,\infty}$ function, and addtional assumption on $f$ is required (e.g. $f$ is H\"older continuous). In that case, $D^{2}w$ is bounded and, therefore, the linearized Monge-Amp\`ere operator for such $w$ will become a uniformly elliptic operator.
\end{remark}

\end{document}